\documentclass[12pt,a4paper]{amsart}
\usepackage{amsfonts}
\usepackage{amsthm}
\usepackage{amsmath, amscd, amssymb, bm}
\usepackage[latin2]{inputenc}
\usepackage{t1enc}
\usepackage[mathscr]{eucal}
\usepackage{indentfirst}
\usepackage{graphicx}
\usepackage{graphics, tikz}
\usepackage{hyperref}
\usepackage{pict2e}
\usepackage{epic}
\numberwithin{equation}{section}
\usepackage[margin=2.9cm]{geometry}
\usepackage{epstopdf}

\DeclareMathOperator{\ming}{\mathcal{M}}
\DeclareMathOperator{\reg}{reg}
\DeclareMathOperator{\pd}{pd}
\DeclareMathOperator{\im}{im}
\DeclareMathOperator{\ind}{Ind}
\DeclareMathOperator{\sdeg}{deg}

\theoremstyle{plain}
\newtheorem{theorem}{Theorem}[section]
\newtheorem{lemma}[theorem]{Lemma}
\newtheorem{cor}[theorem]{Corollary}
\newtheorem{proposition}[theorem]{Proposition}

\theoremstyle{definition}

\newtheorem{remark}[theorem]{Remark}

\newtheorem{example}[theorem]{Example}

\begin{document}

\title{On the cover ideals of chordal graphs}

\author[N. Erey]{Nursel Erey}

\address{Gebze Technical University \\ Department of Mathematics \\
41400 Kocaeli \\ Turkey} 

\email{nurselerey@gtu.edu.tr}

\thanks{Research supported by T\"{U}B\.{I}TAK grant no 118C033}



\begin{abstract} 
 The independence complex of a chordal graph is known to be shellable due to a result of Van Tuyl and Villarreal \cite{vv}. This is equivalent to the fact that cover ideal of a chordal graph has linear quotients. We use this result to obtain recursive formulas for the Betti numbers of cover ideals of chordal graphs. Also, we give a new proof of their result which yields different shellings of the independence complex.

\end{abstract}

\maketitle
%--------------------------------------------------
%               Introduction
%--------------------------------------------------

\section{Introduction}

Let $G$ be a finite simple graph with the vertex set $\{x_1,\dots ,x_n\}$ and let $S=\Bbbk[x_1,\dots ,x_n]$ denote the polynomial ring over some field $\Bbbk$. The edge ideal of $G$, denoted by $I(G)$, is a quadratic squarefree monomial ideal generated by $x_ix_j$ where $\{x_i,x_j\}$ is an edge of $G$. The cover ideal of $G$ is defined by
$$J(G)=(x_{i_1}\dots x_{i_k}: \{x_{i_1},\dots ,x_{i_k}\} \text{ is a minimal vertex cover of } G).$$

Edge and cover ideals have been extensively studied in the literature. The class of chordal graphs arose particularly in the study of such ideals. For instance, according to a celebrated result of Fr\"{o}berg \cite{froberg} the edge ideal of a chordal graph has a linear minimal free resolution if and only if its complement graph is chordal. The authors of \cite{ha van tuyl} obtained recursive formulas for the Betti numbers of edge ideals of chordal graphs and characterized chordal graphs whose edge ideals have linear minimal free resolutions. Such recursive formulas were also used in \cite{kimura f vector} to relate Betti sequences of edge ideals of chordal graphs to $f$-vectors of simplicial complexes.

Francisco and Van Tuyl \cite{francisco van tuyl} proved that if $G$ is a chordal graph, then $S/I(G)$ is sequentially Cohen-Macaulay. Van Tuyl and Villarreal \cite{vv} proved the stronger result that chordal graphs have shellable independence complexes which is equivalent to the statement that cover ideals of chordal graphs have linear quotients. We use this result to obtain recursive formulas for the Betti numbers of cover ideals of chordal graphs which are analogous to those in \cite{ha van tuyl}. As an application of this recursive formula we show that the regularity of the edge ideal of a chordal graph is one more than the induced matching number of the graph (which was originally proved in \cite{zheng} and recovered by other authors \cite{ha van tuyl, kimura, woodroofe}). For some family of chordal graphs (Section \ref*{unmixed section}) our recursive procedure gives exact formulas for the Betti numbers.

In Section \ref{sec:shellings} we recover the result that $J(G)$ has linear quotients when $G$ is chordal and, our proof is based on improving the arguments in \cite{francisco van tuyl}. This yields different shellings of the independence complex than those in \cite{vv}.

%---------------------------------------------
%     Definitions and notation
%--------------------------------------------
\section{Definitions and Notations}

Let $G=(V(G), E(G))$ be a finite simple graph. We call two vertices $u$ and $v$ \textbf{neighbors} if they are adjacent. The \textbf{neighborhood} of $v$, denoted by $N(v)$, is the set of all neighbors of $v$. The \textbf{closed neighborhood} of $v$, denoted by $N[v]$, is $N(v)\cup \{v\}$. A \textbf{complete} graph (or clique) on $n$ vertices is denoted by $K_n$.

For any $A\subseteq V(G)$ the graph $G\setminus A$ stands for the subgraph obtained from $G$ by removing the vertices in $A$.

A graph is \textbf{chordal} if it has no induced cycles of length greater than $3$.  A \textbf{simplicial elimination ordering} in a graph is an ordering of the vertices of the graph such that, for each vertex $v$, $v$ and the neighbors of $v$ that come after $v$ in the order induce a clique. Due to a well-known result of Dirac \cite{dirac} a graph is chordal if and only if it has a simplicial elimination ordering.

 A \textbf{matching} of $G$ is a set of pairwise disjoint edges of $G$. A matching $\{e_1,\dots ,e_k\}$ is called \textbf{induced matching} if $\{u,v\}$ is not an edge of $G$ whenever $u\in e_i$, $v\in e_j$ and $i\neq j$. The \textbf{induced matching number} of $G$, denoted by $\im(G)$, is the maximum cardinality of an induced matching of $G$.

 A \textbf{vertex cover} $C\subseteq V(G)$ of $G$ is a set of vertices such that $C\cap e\neq \emptyset$ for every edge $e\in E(G)$. A vertex cover is called minimal if no proper subset of it is a vertex cover. We will denote the set of \textbf{minimal vertex covers }of $G$ by $\ming(G)$. If $G$ has no edges, then we set $\ming(G)=\{\emptyset\}$. An \textbf{independent} set of $G$ is a set vertices which contains no edges. We call an independent set \textbf{maximal} if it cannot be extended to a bigger independent set.
 
Given a simplicial complex $\Delta$, the \textbf{dimension} of a face $F$ is $\dim(F)=|F|-1$. The dimension of $\Delta$, denoted by $\dim(\Delta)$, is the maximum dimension of all its faces.  A maximal face of $\Delta$ is called a \textbf{facet}. A \textbf{free vertex} of $\Delta$ is a vertex that belongs to exactly one facet of $\Delta$.  
\textbf{Independence complex} of a graph $G$, denoted by $\ind(G)$, is a simplicial complex whose facets are maximal independent sets of $G$. The \textbf{clique complex} of $G$ denoted by $\Delta(G)$, is a simplicial complex whose faces correspond to cliques of $G$.

A simplicial complex $\Delta$ is called \textbf{shellable} if there exists an order $F_1,\dots , F_k $ on the facets of $\Delta$ such that for all $1\leq i<j\leq k$, there exists a vertex $u\in F_j\setminus F_i$ and some $\ell\in\{1,\dots ,j-1\}$ with $F_j\setminus F_\ell =\{u\}$. In such case, we call the order $F_1,\dots , F_k $ a \textbf{shelling} of $\Delta$.

Let $G$ be a graph with the vertex set $\{x_1,\dots ,x_n\}$ and let $S=\Bbbk[x_1,\dots ,x_n]$ for some field $\Bbbk$. The \textbf{edge ideal} of $G$, denoted by $I(G)$, is defined by
$$I(G)=(x_ix_j : \{x_i,x_j\} \text{ is an edge of } G). $$

The \textbf{cover ideal} of $G$, denoted by $J(G)$, is defined by 
$$J(G)=(x_{i_1}\dots x_{i_k}: \{x_{i_1},\dots ,x_{i_k}\} \text{ is a minimal vertex cover of } G).$$

\textbf{Graded Betti numbers} of a monomial ideal $I$ are denoted by $b_{i,j}(I)$. The projective dimension of $I$ is defined by 
$$\pd(I)=\max\{i: b_{i,j}(I)\neq 0 \text{ for some } j \}  $$
and the \textbf{regularity} of $I$ is defined by
$$\reg(I)=\max\{j: b_{i,i+j}(I)\neq 0 \text{ for some } i \}.  $$
For further details on these definitions the readers can refer to \cite{hh book}.
A monomial ideal $I\subset \Bbbk[x_1,\dots ,x_n]$ is said to have \textbf{linear quotients} if there is an ordering $m_1,\dots ,m_q$ on the minimal monomial generators of $I$ such that for every $i=2,\dots ,q$ the ideal $(m_1,\dots ,m_{i-1}):m_i$ is generated by a subset of $\{x_1,\dots ,x_n\}$. In such case we say $m_1,\dots, m_q$ is a \textbf{linear quotients ordering}.

To simplify the notation we will use sets of vertices with monomials interchangeably. 
A squarefree monomial $m$ will substitute for the set $\{x_i: x_i| m\}$. Similarly, a set of vertices $A=\{x_{i_1},\dots ,x_{i_k}\}$ will substitute for the monomial $x_{i_1}\dots x_{i_k}$.

%------------------------------------------------------------
%     A new shelling of independence complex
%------------------------------------------------------------
\section{Shellings of independence complex of a chordal graph}\label{sec:shellings}
The complement of an independent set of a graph $G$ is a vertex cover. Therefore, 
$$\ind(G)=\langle F_1,\dots ,F_k \rangle \Longleftrightarrow J(G)=(F_1^c,\dots ,F_k^c).$$ 

Observe that $F_1,\dots ,F_k$ is a shelling of $\ind(G)$ if and only if $F_1^c,\dots ,F_k^c$ is a linear quotients ordering for $J(G)$. Van Tuyl and Villarreal \cite{vv} proved that independence complex of a chordal graph is shellable \cite[Theorem 2.13]{vv}. Analyzing inductive argument of their proof, one can deduce the following result.
% Also in page 180 of monomial ideals book

\begin{theorem}\cite{vv}\label{thm:vv linear quotients} Let $G$ be a chordal graph with a vertex $x_1$ such that $N[x_1]=\{x_1,\dots,x_r\}$ induces a clique for some $r\geq 2$. For each $i=1,\dots ,r$, let $G_i=G\setminus N[x_i]$ and $\displaystyle y_i=\prod_{x_j\in N(x_i)} x_j$. Suppose that for each $i$, $u^i_1,\dots ,u^i_{s_i} $ is a linear quotients ordering for $J(G_i)$. Then
	$$y_1 u^1_1 ,\dots ,y_1 u^1_{s_1}; y_2 u^2_1 ,\dots ,y_2 u^2_{s_2};\cdots ; y_r u^r_1 ,\dots ,y_r u^r_{s_r} $$
	is a linear quotients ordering for $J(G)$. Note here that if $G_i$ has no edges, then the sequence $y_i u^i_1 ,\dots ,y_i u^i_{s_i}$ is just $y_i$.
\end{theorem}

%---------------------------------
% Complete graph
%---------------------------------

\begin{proposition}\label{complete graph}
Let $K_n$ be a complete graph on at least $2$ vertices. Then $J(G)$ has linear quotients with respect to any ordering of its minimal generators and thus has a linear resolution. Moreover, $b_{0,n-1}(J(G))=n$, $b_{1,n}(J(G))=n-1$ and $\pd(J(G))=1$.
\end{proposition}
\begin{proof}
Observe that $J(G)$ is generated in degree $n-1$ and then the given formulas follow from Theorem \ref{thm:betti number of linear quotients}.
\end{proof}

\begin{lemma}\label{lem: mvc of chordal}
	Let $G$ be a chordal graph with a vertex $v$ such that $N[v]$ induces a clique $K_r$ for some $r\geq 1$. Let $w=N(v)$, $H_1=G\setminus \{v\}$ and $H_2=G\setminus N[v]$. Then
	$$\ming (G)=\{w\cup A: A\in\ming(H_2)\}\cup \{\{v\}\cup B: B\in \ming(H_1) \text{ and } w\nsubseteq B\}.$$
\end{lemma}
	\begin{proof} First observe that the equality holds if $v$ has no neighbors or $v$ is the only vertex of $G$ as we set $\mathcal{M}(H)=\{\emptyset\}$ for a graph $H$ with no edges. So, let us assume that $r\geq 2$.
		
		($\subseteq$): Let $u$ be a minimal vertex cover of $G$. First observe that since $K_r$ is a complete subgraph, $u$ contains at least $r-1$ vertices of $K_r$. On the other hand, $u$ cannot contain all of the vertices of $K_r$ as it would make $v$ redundant in the cover. Therefore, $u$ contains exactly $r-1$ vertices of $K_r$. We consider cases:
		
		Case 1: Suppose $v\in u$. Then $w\nsubseteq u$ as $u$ cannot contain all the vertices of $K_r$. Then $u\setminus \{v\}$ is a minimal vertex cover of $H_1$.
		
		Case 2: Suppose $v\notin u$. Then $w\subseteq u$. Observe that $u\setminus w$ must be a minimal vertex cover of $H_2$.
		
		($\supseteq$): One can show that every minimal vertex cover of $H_2$ can be extended to that of $G$ by adding the neighbors of $v$. Similarly, every minimal vertex cover $B$ of $H_1$ can be extended to that of $G$ by adding $v$ provided $w$ is not contained in $B$. 
	\end{proof}
	Francisco and Van Tuyl \cite{francisco van tuyl} proved that the cover ideal $J(G)$ of a chordal graph $G$ is componentwise linear. Their proof is based on showing that the ideal $(J(G)_d)$ generated by all degree $d$ elements of $J(G)$ has linear quotients for all $d$. We improve their arguments by dealing with the minimal generators of $J(G)$ instead of those of $(J(G)_d)$ and recover the following result.
	 
	\begin{theorem}\cite{vv}\label{thm: linear quotient of chordal graphs}
		If $G$ is a chordal graph, then $J(G)$ has linear quotients. 
	\end{theorem}
	\begin{proof}
		We proceed by induction on the number of vertices of $G$. We may assume that $G$ is not complete since otherwise the result follows from Proposition \ref{complete graph}. Let $x$ be a vertex of $G$ with $N(x)=\{y_1,\dots ,y_t\}$ such that the induced subgraph on $N[x]$ is complete. Let $H_1=G\setminus \{x\}$ and $H_2=G\setminus N[x]$.
		
		Let $A_1, \dots , A_a$ be the minimal generators of $J(H_2)$ and let $B_1, \dots , B_b$ be the minimal generators of $J(H_1)$ both written in the linear quotients ordering.
		
		Let $B_{i_1}, \dots , B_{i_k}$ be the minimal vertex covers of $H_1$ which do not contain $N(x)$ where $i_1<\dots <i_k$. By Lemma~\ref{lem: mvc of chordal} 
		$$ yA_1, \dots , yA_a, xB_{i_1}, \dots , xB_{i_k}$$
		are the minimal generators of $J(G)$ where $y=y_1\dots y_t$. We claim that the list above is a linear quotients ordering. By induction assumption $(yA_1,\dots , yA_{i-1}):(yA_i)$ is generated by variables for $i=2,\dots ,a$. Without loss of generality, suppose that $y_t \notin B_{i_1}$. First, we claim that 
		$$ (yA_1, \dots , yA_a): xB_{i_1} =(y_t).$$
		Observe that $B_{i_1}\setminus \{y_1,\dots , y_{t-1}\}$ is a vertex cover of $H_2$. Then $A_j\subseteq B_{i_1}\setminus \{y_1,\dots , y_{t-1}\}$ for some $j$ and we get $yA_j\setminus xB_{i_1}= \{y_t\}$.
		
		Let $2\leq r \leq k$ be fixed. We will show that $(yA_1, \dots , yA_a, xB_{i_1}, \dots , xB_{i_{r-1}}):xB_{i_r}$ is generated by variables. Suppose that $y_q\notin B_{i_r}$ for some $1\leq q \leq t$. From the argument above we get 
		$$(yA_1, \dots , yA_a, xB_{i_1}, \dots , xB_{i_{r-1}}):xB_{i_r}= (y_q) + (xB_{i_1}, \dots , xB_{i_{r-1}}):xB_{i_r}.$$
		
		Let $1\leq s \leq r-1$ be fixed. Suppose that $B_{i_s}\setminus B_{i_r}$ has at least two vertices. By induction assumption, there exists $1 \leq \ell < i_s$ such that $B_{\ell}\setminus B_{i_r}=\{z\}$ for some $z\in B_{i_s}\setminus B_{i_r}$. If $\ell \in \{i_1,\dots , i_{s-1}\}$, then $z$ is a generator of $(xB_{i_1}, \dots , xB_{i_{r-1}}):xB_{i_r}$. Let us assume that $\ell \notin \{i_1,\dots , i_{s-1}\}$. Then $\{y_1,\dots , y_t\}\subseteq B_{\ell}$ and $z=y_q$.
	\end{proof}
	
	The inductive proof of the theorem above describes how to construct linear quotients ordering recursively using subgraphs.
	
\begin{cor}\label{cor:my order}
	Suppose $G$ is a chordal graph and $x$ is a vertex of $G$ such that $N[x]$ induces a clique on at least $2$ vertices. Let $A_1, \dots , A_a$ be the minimal generators of $J(G\setminus N[x])$ and let $B_1, \dots , B_b$ be the minimal generators of $J(G\setminus \{x\})$ both written in the linear quotients ordering. Let $B_{i_1}, \dots , B_{i_k}$ be the generators which do not contain $N(x)$ where $i_1<\dots <i_k$. Then 
	$$ N(x)A_1, \dots , N(x)A_a, xB_{i_1}, \dots , xB_{i_k}   $$
	is a linear quotients ordering for $J(G)$.
\end{cor}
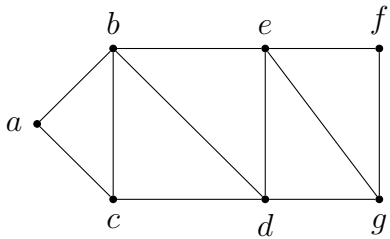
\begin{figure}[hbt!]
 \centering
\begin{tikzpicture}
[scale=1.00, vertices/.style={draw, fill=black, circle, minimum size = 4pt, inner sep=0.5pt}, another/.style={draw, fill=black, circle, minimum size = 2.5pt, inner sep=0.1pt}]
						\node[another, label=left:{$a$}] (a) at (-2,0) {};
						\node[another, label=above:{$b$}] (b) at (-1, 1) {};
						\node[another, label=below:{$c$}] (c) at (-1, -1) {};
						\node[another, label=below:{$d$}] (d) at (1,-1) {};
						\node[another, label=above:{$e$}] (e) at (1,1) {};
						\node[another, label=above:{$f$}] (f) at (2.5, 1) {};
				           \node[another, label=below:{$g$}] (g) at (2.5, -1) {};
						\foreach \to/\from in {a/b, a/c, b/c, b/d, b/e, c/d, d/e, d/g, e/f, e/g, f/g}
						\draw [-] (\to)--(\from);
						\end{tikzpicture}
						 \caption{\label{graph} A chordal graph}
\end{figure}
\begin{example}[\textbf{Comparison of shellings}]
	Let $G$ be the graph in Figure \ref{graph}. Then $N[a]$ induces a clique on $3$ vertices. Then $J(G\setminus N[a])=(eg,fdg,fed)$, $J(G\setminus N[b])=(f,g)$ and $J(G\setminus N[c])=(ef,fg,ge)$ where in each case the generators are listed in a linear quotients ordering. According to Theorem \ref{thm:vv linear quotients} if we take the neighbors of $a$ in the order $b, c$ we get the linear quotients ordering
	$$\bm{bc}eg, \bm{bc}fdg, \bm{bc}fed, \bm{aced}f, \bm{aced}g, \bm{abd}ef, \bm{abd}fg, \bm{abd}ge $$
	of $J(G)$. And, this gives the shelling $adf, ae, ag, bg, bf, cg, ce, cf $ of $\ind(G)$. Similarly, if we take the neighbors of $a$ in the order $c, b$ we get the linear quotients ordering
	$$\bm{bc}eg, \bm{bc}fdg, \bm{bc}fed, \bm{abd}ef, \bm{abd}fg, \bm{abd}ge, \bm{aced}f, \bm{aced}g $$
of $J(G)$ which gives the shelling $adf, ae, ag, cg, ce, cf, bg, bf$ of $\ind(G)$.

 Now let us construct a linear quotients ordering using Corollary \ref{cor:my order}. Observe that $J(G\setminus\{a\})=(bceg, cdeg,bdeg,bdef,cdef,bdgf)$ with generators listed in a linear quotients ordering. The generator $bceg$ is the only minimal vertex cover of $G\setminus \{a\}$ that contains $N(a)$. Therefore the list
$$\bm{bc}eg, \bm{bc}fdg, \bm{bc}fed, \bm{a}cdeg, \bm{a}bdeg, \bm{a}bdef, \bm{a}cdef, \bm{a}bdgf $$
is a linear quotients ordering for $J(G)$ which corresponds to the shelling $adf$, $ae$, $ag$, $bf$, $cf$, $cg$, $bg$, $ce$ of $\ind(G)$.
\end{example}
\section{Recursive computation of Betti numbers}
In this section we will find some recursive formulas for Betti numbers of cover ideals of chordal graphs. We will need the following theorem which describes Betti numbers of ideals with linear quotients.
%------------------------------------------------------
%       graded betti numbers
%------------------------------------------------------
% complete oldugu durum OK? check it
%  The given formula does not hold for S/I

\begin{theorem}\cite[Corollary 2.7]{sv}\label{thm:betti number of linear quotients}
Let $I$ be a homogeneous ideal with linear quotients with respect to $f_1,\dots ,f_m$ where $\{f_1,\dots ,f_m\}$ is a minimal system of homogeneous generators for $I$. Let $n_p$ be the minimal number of homogeneous generators of $(f_1,\dots ,f_{p-1}):f_p$ for $p=1,\dots ,m$. Then
$$ b_{i, i+j}(I)=\sum_{\substack{1\leq p \leq m\\ \sdeg(f_p)=j}} {n_p \choose i}$$

$$ b_i(I)=\sum_{p=1}^m {n_p \choose i}.$$
\end{theorem}

\begin{theorem}[\textbf{Recursive formula for graded Betti numbers}]\label{thm:graded recursive}
Let $G$ be a chordal graph with a vertex $x_1$ such that $N[x_1]=\{x_1,\dots, x_r\}$ induces a clique. Let $H_i$ be the graph obtained from $G$ by removing the closed neighborhood of $x_i$. Then for all $i\geq 1$
$$  b_{i,i+j}(J(G))= \sum_{t=1}^{r} b_{i, i+j-|N(x_t)|}(J(H_t)) + \sum_{t=2}^{r}b_{i-1,i+j-|N(x_t)|}(J(H_t)) $$
where if $H$ has no edges we set $b_{0,1}(J(H))=1$ and $b_{p,q}(J(H))=0$ when $(p,q)\neq (0,1)$.
\end{theorem}
\begin{proof}
	
Suppose that for each $t$, the ideal $J(H_t)$ has linear quotients with respect to the order $u^t_1,\dots ,u^t_{s_t}$ of its minimal generators. Then 
	 $$f_1,f_2,\dots ,f_q:=y_1 u^1_1 ,\dots ,y_1 u^1_{s_1}; y_2 u^2_1 ,\dots ,y_2 u^2_{s_2};\cdots ; y_r u^r_1 ,\dots ,y_r u^r_{s_r} $$
	 is a linear quotients ordering for $J(G)$ by Theorem~\ref{thm:vv linear quotients}. Clearly, for all $1\leq p\leq s_1$
	 \begin{equation}\label{eq:f1}
	(f_1,\dots,f_{p-1}):f_p=(u_1^1,\dots , u^1_{p-1}):u_{p-1}^1.
	 \end{equation}
	 Also, for every $t\geq 2$, any minimal vertex cover of $H_t$ can be extended to that of $H_1$ by adding some neighbors of $x_t$ since $N[x_1]\subseteq N[x_t]$. Therefore
	 if $f_p=y_tu_{\ell}^t$ for some $t\geq 2$ and $1\leq \ell\leq s_t$, then we have
	 \begin{equation}\label{eqref:fp}
	 (f_1,\dots ,f_{p-1}):f_p=(x_t)+(u_1^t, u_2^t, \dots , u_{\ell-1}^t):u_\ell^t. 
	 \end{equation}
	 
	 Let $n_p$ be the minimal number of homogeneous generators of $(f_1,\dots ,f_{p-1}):f_p$ for all $p=1,\dots, q$. Also, for every $t\geq 1$ and $1\leq \ell\leq s_t$ let $m_\ell^t$ be the minimal number of homogeneous generators of $(u_1^t, u_2^t, \dots ,u_{\ell-1}^t):u_\ell^t$. Then using Theorem \ref{thm:betti number of linear quotients} we have 
	  \begin{equation*}
	 \begin{split}
	 b_{i,i+j}(J(G)) & =   \sum_{\sdeg(f_p)=j}{n_p \choose i}
 \\
 & = \sum_{\substack {\sdeg (f_p)=j \\ p\leq s_1}}{{n_p}\choose{i}} +\sum_{\substack {\sdeg (f_p)=j \\ p> s_1}}{{n_p}\choose{i}}\\
  & =  \sum_{\substack {\sdeg (u_p^1)=j-|N(x_1)| \\ p\leq s_1}}{{m_p^1}\choose{i}} +\sum_{t=2}^{r}\sum_{\substack{1\leq \ell \leq s_t\\ \sdeg(u_\ell^t)=j-|N(x_t)|}}{m_\ell^t+1\choose i}\quad  \text{ by \eqref{eq:f1} and \eqref{eqref:fp}}\\
	 & = b_{i,i+j-|N(x_1)|}(J(H_1))+ \sum_{t=2}^{r}\sum_{\substack{1\leq \ell \leq s_t\\ \sdeg(u_\ell^t)=j-|N(x_t)|}} {m_\ell^t\choose i-1}+{m_\ell^t\choose i} \\
	 & = \sum_{t=1}^{r} b_{i, i+j-|N(x_t)|}(J(H_t)) + \sum_{t=2}^{r}b_{i-1,i+j-|N(x_t)|}(J(H_t))
	 \end{split}
	 \end{equation*}
	 as desired.
\end{proof}
%----------------------------------------------------
%       total betti numbers
%---------------------------------------------------

\begin{theorem}[\textbf{Recursive formula for total Betti numbers}]\label{thm:recursive formula}
Let $G$ be a chordal graph with a vertex $x_1$ such that $N[x_1]=\{x_1,\dots, x_r\}$ induces a clique. Let $H_i$ be the graph obtained from $G$ by removing the closed neighborhood of $x_i$. Then for all $i\geq 1$,
$$b_i(J(G)) =\sum\limits_{t=1}^{r}b_i(J(H_t))+\sum_{t=2}^{r}b_{i-1}(J(H_t))$$
where if $H$ has no edges we set $b_0(J(H))=1$ and $b_p(J(H))=0$ for any $p\neq 0$.
\end{theorem}

%The given formula does not hold for $S/J!$
%Note: In this formula $b_0(J(K))=1$ if $K$ is an isolated vertex, or an empty graph i.e., $b_1(S/J(K))=1$. Must be careful about this in the inductive step, otherwise the formula is wring!! check the path of lengths 3 and 4. the formula does not hold for complete graphs i think
\begin{proof}
	Let $f_1,\dots, f_q$ be as in the proof of Theorem \ref{thm:graded recursive}. Let $n_p$ be the minimal number of homogeneous generators of $(f_1,\dots ,f_{p-1}):f_p$ for all $p=1,\dots, q$. Also, for every $t\geq 2$ and $1\leq j\leq s_t$ let $m_j^t$ be the minimal number of homogeneous generators of $(u_1^t, u_2^t, \dots ,u_{j-1}^t):u_j^t$. Then using Theorem \ref{thm:betti number of linear quotients} we have the following equation.
	 \begin{equation*}
	 \begin{split}
	 b_i(J(G)) & = \sum\limits_{p=1}^{q}{{n_p}\choose{i}} \\
	 & = \sum_{p=1}^{s_1} {{n_p}\choose {i}} +  \sum_{p=s_1+1}^{q} {{n_p}\choose {i}}\\
	 & = b_i(J(H_1)) + \sum_{t=2}^r\sum_{j=1}^{s_j}{m_j^t+1\choose i} \quad \text{by \eqref{eq:f1} and \eqref{eqref:fp}}\\
	 & = b_i(J(H_1)) + \sum_{t=2}^r\sum_{j=1}^{s_j}{m_j^t\choose i-1}+{m_j^t\choose i}\\
	 & = b_i(J(H_1))+\sum\limits_{t=2}^{r}b_i(J(H_t))+\sum_{t=2}^{r}b_{i-1}(J(H_t))
	 \end{split}
	 \end{equation*}
	 as desired.
\end{proof}

We can apply the theorem above to the path graphs which form a subclass of chordal graphs.

\begin{cor} Let $P_n$ denote a path on $n$ vertices for $n\ge 4$. Then for every $i\geq 1$
		$$b_i(J(P_n))=b_i(J(P_{n-2}))+b_i(J(P_{n-3}))+b_{i-1}(J(P_{n-3})). $$
	\end{cor}
	\begin{proof}
		Let $x_1x_2, x_2x_3,\dots , x_{n-1}x_n$ be the edges of $P_n$. Then $N[x_1]=\{x_1,x_2\}$ induces a clique on $2$ vertices. As $G\setminus N[x_1]=P_{n-2}$ and $G\setminus N[x_2]=P_{n-3}$, the result follows immediately. 
	\end{proof}
	
	\begin{remark}
	For any path $P_n$ on $n\geq 1$ vertices, the Betti number $b_0(J(P_n))$ is the number of minimal vertex covers of of $P_n$. By Theorem \ref{thm:vv linear quotients} for any $n\geq 4$ we have the recursive formula
	$$b_0(J(P_n))=b_0(J(P_{n-2}))+b_0(J(P_{n-3})). $$
	
For any $n\geq 3$ the sequence defined by the recurrence relation
$$p_n=p_{n-2}+p_{n-3}, \quad  \ p_0=p_1=p_2=1 $$
is known as the Padovan sequence. Therefore the sequence $\{b_0(J(P_n))\}_{n\geq 1}$ of Betti numbers coincides with the sequence $\{p_n\}_{n\geq 2}$.
	\end{remark}
%----------------------------------------------------
%  projective dimension
%---------------------------------------------------
Due to a well known result of Terai \cite{terai} regularity of a squarefree monomial ideal is related to the projective dimension of its Alexander dual.

\begin{theorem}\cite{terai}\label{terai}
Let $I\subseteq S$ be a squarefree monomial ideal. Then
$$\reg(I)=\pd(S/I^{\vee}).  $$
	\end{theorem}

As another consequence of Theorem \ref{thm:recursive formula} we obtain a new proof of the following result.

	\begin{cor}\cite{ha van tuyl, kimura, woodroofe, zheng}
		If $G$ is a chordal graph, then 
\begin{enumerate}
\item $\pd(J(G))=\im(G)$,
\item  $\reg(S/I(G))=\im(G)$.
\end{enumerate}
		\end{cor}
	\begin{proof}
	By Theorem \ref{terai} it suffices to prove (1) as $I(G)^\vee=J(G)$. We proceed by induction on the number of vertices of the graph. First note that the result is clear if $G$ has no edges. Also, we may assume that $G$ has no isolated vertices as they do not affect the cover ideal. So, let us assume that $G$ has a vertex $x_1$ such that $N[x_1]=\{x_1,\dots ,x_r\}$ induces a clique for some $r\geq 2$. Let $$\kappa=\max\{  \im(H_1), \max\{\im(H_i)+1: i=2,\dots ,r\}          \}$$ where $H_i$ are as in the statement of the Theorem \ref{thm:recursive formula}. By induction assumption and Theorem \ref{thm:recursive formula} it suffices to show that $\kappa=\im(G)$. It is clear that $\im(H_1)\leq\im(G)$ since $H_1$ is an induced subgraph of $G$. Also, any induced matching of $H_i$ where $i\geq 2$ can be extended to an induced matching of $G$ by adding the edge $x_1x_i$. Therefore, $\kappa\leq \im(G)$. 

Conversely, let $A$ be an induced matching of $G$ of maximum cardinality. If there exists an edge $e\in A$ such that $x_i\in e$ for some $i\geq 2$, then removing $e$ from $A$ yields an induced matching for $H_i$ and, $\im(G)\leq \im(H_i)+1$. Otherwise, $A$ is an induced matching of $H_1$ and $\im(G)\leq \im(H_1)$. Thus in both cases $\im(G)\leq \kappa$.
	\end{proof}

\subsection{Unmixed chordal graphs with $1$-dimensional independence complexes}\label{unmixed section}

% not needed:
%\begin{lemma}\cite[Lemma 6.7.12]{vil}\label{lem:vil chordal}
%Let $G$ be a chordal graph, and let $K$ be a complete subgraph of $G$. If $K\neq G$, then there is a vertex $x\notin V(K)$ such that the subgraph induced by the neighbor set $N(x)$ of $x$ is a complete subgraph.
%\end{lemma}

In this section, we demonstrate our recursive formulas on some unmixed chordal graphs. We will use the following characterization of unmixed chordal graphs due to Herzog et al.\cite{hhz chordal}.
\begin{theorem}\cite[Theorem 2.1]{hhz chordal}\label{CM chordal graphs}
	Let $K$ be a field, and let $G$ be a chordal graph on the vertex set $[n]$. Let $F_1,\dots ,F_m$
	be the facets of $\Delta(G)$ which admit a free vertex. Then the following conditions
	are equivalent:
	\begin{itemize}
		\item $G$ is Cohen-Macaulay;
		\item  $G$ is Cohen-Macaulay over $K$;
		\item $G$ is unmixed;
		\item $[n]$ is the disjoint union of $F_1,\dots ,F_m$.
	\end{itemize}
\end{theorem}

\begin{remark}\label{rk:b0}
	For any graph $G$, the number of minimal vertex covers is equal to the number of maximal independent sets. Therefore, $b_0(J(G))$ is the number of facets of $\ind(G)$.
\end{remark}

\begin{proposition}\label{prop:1-dimensional unmixed}
Let $G$ be an unmixed chordal graph on $n$ vertices such that the independence complex of $G$ has dimension $1$. Then 
\begin{equation*}
\begin{split}
b_1(J(G)) & = 2b_0(J(G))-n \\
b_2(J(G)) & = b_0(J(G))-n+1
\end{split}
\end{equation*}
\end{proposition}

\begin{proof}
Let $F_1,\dots ,F_m$ be the facets of $\Delta(G)$ which admit a free vertex. By Theorem~\ref{CM chordal graphs} the vertex set of $G$ is the disjoint union of $F_1,\dots ,F_m$. As $\ind(G)$ has dimension one, $m=2$. Let $F_1=\{a_1,\dots , a_q\}$ and $|F_2|=p$ where $a_1$ is a free vertex. For each $i=2,\dots, q$ let $t_i=|N(a_i)\cap F_2|$. Observe that by Remark~\ref{rk:b0}
\begin{equation}\label{eq:b0}
b_{0}(J(G))= pq-\sum_{i=2}^{q}t_i
\end{equation}

Using Theorem~\ref{thm:recursive formula}, Proposition~\ref{complete graph} and Eq.\ref{eq:b0} we get
\begin{equation*}
\begin{split}
b_1(J(G)) & = b_1(J(G\setminus N[a_1])) + \sum_{i=2}^q b_1(J(G\setminus N[a_i]))+\sum_{i=2}^q b_0(J(G\setminus N[a_i]))  \\
& =  b_1(J(K_p)) + \sum_{i=2}^q b_1(J(K_{p-t_i}))+\sum_{i=2}^q b_0(J(K_{p-t_i}))\\
&= (p-1)+\sum_{i=2}^q(p-t_i-1)+\sum_{i=2}^q(p-t_i) \\
&= 2pq-p-q-2\sum_{i=2}^{q}t_i \\
&= 2b_0(J(G))-n.
\end{split}
\end{equation*}
Similarly,
\begin{equation*}
\begin{split}
b_2(J(G)) & = b_2(J(G\setminus N[a_1])) + \sum_{i=2}^q b_2(J(G\setminus N[a_i]))+\sum_{i=2}^q b_1(J(G\setminus N[a_i]))  \\
& =  b_2(J(K_p)) + \sum_{i=2}^q b_2(J(K_{p-t_i}))+\sum_{i=2}^q b_1(J(K_{p-t_i}))\\
&= 0+\sum_{i=2}^q(p-t_i-1) \\
&= b_0(J(G))-n+1.
\end{split}
\end{equation*}
\end{proof}
 It would be an interesting problem to extend Proposition \ref{prop:1-dimensional unmixed} to unmixed chordal graphs with higher dimensional independence complexes. More generally, characterizing Betti numbers of cover ideals of chordal graphs would be of interest.

\end{document}